\documentclass[12pt]{article}
\usepackage{geometry}                
\usepackage{graphicx}
\usepackage{amssymb}
\usepackage{epstopdf}
\usepackage{graphicx}
\usepackage{amssymb}
\usepackage{epstopdf}
\usepackage{amsmath}
\usepackage{amsthm}
\usepackage{amsfonts}
\usepackage{epstopdf}

\newtheorem{theorem}{Theorem}
\newtheorem{lemma}[theorem]{Lemma}

\newtheorem{corollary}[theorem]{Corollary}

\newtheorem{proposition}[theorem]{Proposition}
\theoremstyle{definition}

\theoremstyle{remark}

\numberwithin{equation}{section}

\newcommand\RR{{\mathbb R}}

\newcommand\EE{{\mathbb E}}
\newcommand\PP{{\mathbb P}}

\def\W2{W^{1,2}({\cal O}(M))}

\newcommand{\Norm}[2]{\left| \left| #2 \right| \right|_{#1}}

\newcommand\eps{\varepsilon}

\title{Small deviations for beta ensembles}

\author{Michel Ledoux 
\and Brian Rider}

\begin{document}
\maketitle

\begin{abstract} 
We establish various small deviation inequalities for the extremal (soft edge)
eigenvalues in the $\beta$-Hermite and $\beta$-Laguerre ensembles.
In both settings, upper bounds on the variance of the largest
eigenvalue of the anticipated order follow immediately.
\end{abstract}

\section{Introduction}

In the context of their original discovery, the Tracy-Widom laws describe the fluctuations of the limiting largest
eigenvalues in the Gaussian Orthogonal, Unitary, and Symplectic Ensembles (G$\{$O/U/S$\}$E) \cite{TW1, TW2}.
These are random matrices of real, complex, or quaternion Gaussian entries, of mean zero and mean-square
one, independent save for the condition that the matrix is symmetric (GOE), Hermitian (GUE), or appropriately
self-dual (GSE). 
The corresponding Tracy-Widom distribution functions have shape
\begin{equation}
\label{TWShape}
   F_{TW}(t) \sim e^{\frac{1}{24}  \beta t^3 } \mbox{    as    } t \rightarrow - \infty,  \   \  1 - F_{TW}(t) \sim e^{- \frac{2}{3} \beta t^{3/2}} \mbox{   as   }  t 
   \rightarrow \infty,
\end{equation}  
where $\beta = 1$ in the case of GOE, $\beta = 2$ for GUE, and $\beta=4$ 
for GSE.

Since that time, it has become understood that
the three Tracy-Widom laws  arise in a wide range of models.
First, the assumption of Gaussian entries may be relaxed significantly, see \cite{Sosh}, \cite{TaoVu}
for instance.  Outside of random 
matrices, these laws also describe the fluctuations in the longest increasing subsequence of a random permutation \cite{BDJ},
the path weight in last passage percolation \cite{J},  and the current in simple exclusion \cite{J, TW3}, among others.

It is natural to inquire as to the rate of concentration of these various objects about the limiting Tracy-Widom laws.  Back in the 
random matrix setting, the limit theorem reads: with $\lambda_{\max}$ the largest eigenvalue in the $n \times n$ GOE, GUE or 
GSE, it is the 
normalized quantity $n^{1/6}(\lambda_{\max} - 2 \sqrt{n} )$ which converges to Tracy-Widom.  Thus,  one would optimally hope
for estimates of the form: 
$$
  \PP \Bigl( \lambda_{\max} - 2 \sqrt{n} \le - \eps \sqrt{n} \Bigr) \le C e^{- n^2 \eps^{3}/C},  \  \ 
  \PP \Bigl( \lambda_{\max} - 2 \sqrt{n} \ge  \eps \sqrt{n} \Bigr) \le C e^{- n \eps^{3/2}/C}, 
$$
for all $n \ge 1$, all $\eps \in (0,1]$ say, and $C$ a numerical constant.   Such are ``small deviation" inequalities, capturing
exactly the finite $n$ scaling and limit distribution shape (compare (\ref{TWShape})). 
Taking $\eps$ beyond $O(1)$ in the
above yields more typical large deviation behavior and different (Gaussian) tails (see below).

As discussed in \cite{L1, L2}, the right-tail inequality for the GUE (as well as for the
Laguerre Unitary Ensemble, again see below) may be shown to follow from results of Johansson \cite {J}
for a more general invariant model related to the geometric distribution that uses
large deviation asymptotics and sub-additivity arguments. 
The left-tail inequality for the geometric model of Johansson (and thus by some suitable
limiting procedure for the GUE and the Laguerre Unitary Ensemble) is established in
\cite {BDMLMZ} together with convergence of moments using delicate Riemann-Hilbert
methods. We refer to \cite{L2} for a discussion and the relevant references, as
well as for similar inequalities in the context of last passage percolation {\em{etc}}.
By the superposition-decimation procedure of \cite {FR}, the GUE bounds apply similarly
to the GOE (see also \cite {L3}).

Our purpose here is to present unified proofs of these bounds which apply to all
of the so-called beta ensembles.  These
are point-processes on $\RR$ defined by the $n$-level joint density: for any $\beta > 0$, 
\begin{equation}
\label{betadens}
 \PP(\lambda_1, \lambda_2, \dots, \lambda_n) =  \frac{1}{Z_{n,
     \beta}}  \prod_{j < k} |
\lambda_j - \lambda_k |^{\beta}  e^{ - (\beta/4)  \sum_{k=1}^n \lambda_k^2}.
\end{equation} 
At $\beta = 1, 2, 4$ this joint density is shared by the eigenvalues of G$\{$O/U/S$\}$E.  Furthermore, these three
values give rise to exactly solvable models.  Specifically, all finite dimensional correlation functions may be described
explicitly in terms of Hermite polynomials.  For this reason, the measure (\ref{betadens}) has come to be referred to the
$\beta$-Hermite ensemble; we will denote it by $H_{\beta}$.  Importantly, off of $\beta = 1,2, 4$,
despite considerable efforts (see \cite{For}, Chapter 13 for a comprehensive review), 
there appears to be no characterization of the correlation functions amenable to asymptotics. Still,
Ram\'irez-Rider-Vir{\'a}g \cite{RRV} have shown the existence of a general $\beta$ Tracy-Widom law, $TW_{\beta}$,
via the corresponding limit theorem:  with self-evident notation,
\begin{equation}
\label{HLimit}
     n^{1/6} \big ( \lambda_{\max}(H_{\beta}) - 2 \sqrt{n} \, \big ) \Rightarrow TW_{\beta}.
\end{equation}
This result makes essential use of a (tridiagonal) matrix model valid at all beta due to Dumitriu-Edelman \cite{DE}, and
proves the conjecture of Edelman-Sutton \cite{ES}.  As to finite $n$ bounds, we have the following.

\begin{theorem}
\label{Thm-Hermite}
For all $n \ge 1 $,  $0 < \eps \le 1$ and $\beta \ge 1$$:$
\begin{equation*}
     \PP \Bigl( \lambda_{\max}(H_{\beta})  \ge 2  \sqrt{n} (1+ \eps)  \Bigr)  \le  {C} e^{- \beta n \eps^{3/2}/C},
\end{equation*}
and 
\begin{equation*} 
     \PP \Bigl( \lambda_{\max}(H_{\beta})  \le   2 \sqrt{n} (1 -\eps)    \Bigr)  \le  {C^{\beta}} e^{- \beta n^2 \eps^{3}/C},
\end{equation*}
where $C$ is a numerical constant.
\end{theorem}

The restriction to $\beta \ge 1$ is somewhat artificial.  On the other hand, bounds of this type cannot remain meaningful
all the way down to $\beta = 0$.  Our method in fact applies to all beta bounded below, though with the reported $C$ 
 a function of whatever specified minimal beta.  Keeping $\beta \ge 1$ covers the cases of classical interest 
while allowing for a clearer picture of the achieved beta dependence in our estimates (as well as cleaner proofs).

For completeness we also mention that for $\eps$ beyond $O(1)$, the large-deviation right-tail inequality 
takes the form 

\begin{equation}
\label{largedev}
 \PP \Bigl( \lambda_{\max}(H_{\beta})  \ge 2  \sqrt{n} (1+ \eps)  \Bigr)  \le  {C} e^{- \beta n \eps^{2}/C}.
 \end{equation}
For $\beta = 1$ and $ 2$ this follows from standard 
net arguments on the corresponding Gaussian matrices (see e.g.\ \cite{L2}). For other values of $\beta $,
crude bounds on the tridiagonal models discussed below immediately yield the claim.

Continuing, those well versed in random matrix theory will know that this style of small deviation questions
are better motivated in the context of ``null" Wishart matrices, given their application in multivariate statistics.  
Also known as the Laguerre Orthogonal or Unitary Ensembles (L$\{$O/U$\}$E),
these are ensembles
of type $X X^*$ in which $X$ is an $n \times \kappa$ matrix comprised of i.i.d.\ real or complex Gaussians.    

By the obvious
duality, we may assume here that $\kappa \ge n$.     When $n \rightarrow \infty$ with the $\kappa/n$
converging to a finite constant (necessarily larger than one), the appropriately centered and scaled largest eigenvalue
was shown to converge to the natural Tracy-Widom distribution; first by Johansson \cite{J} in the complex ($\beta = 2$) case,
then by Johnstone \cite{JS} in the real ($\beta = 1$) case.  Later, El Karoui \cite{EK} proved the same conclusion allowing
$\kappa/n \rightarrow \infty$.   

For $\beta = 2$ and $\kappa$ a fixed multiple of $n$, 
a small deviation upper bound at the right-tail
(as well as the corresponding statement for the minimal eigenvalue in the ``soft-edge"
scaling) was known earlier (see \cite {L1, L2}), extended recently to non-Gaussian matrices in \cite {FS}.

Once again there is a general beta version. Consider a density of the form (\ref{betadens}) in which the Gaussian weight
$w(\lambda) = e^{-\beta \lambda^2/4} $ on $\RR$ is replaced by $w(\lambda) $ $= \lambda^{(\beta/2) (\kappa - n+1) + 1}$ $e^{- \beta \lambda/2}$, now restricted to $\RR_+$.   
Here $\kappa$ can be any real number strictly larger than $n-1$.  It is when $\kappa$
is an integer and $\beta = 1$ or $2$ that one recovers the eigenvalue law for the real or complex Wishart matrices just described.
For general $\kappa$ and $\beta > 0$ the resulting law on positive points $\lambda_1, \ldots, \lambda_n$ is referred to as the $\beta$-Laguerre ensemble, here $L_{\beta}$ for short.

Using a tridiagonal model for $L_{\beta}$ 
introduced in \cite{DE}, it is proved in \cite{RRV}: for $\kappa +1  >  n  \rightarrow \infty$ with $\kappa/n \rightarrow c \ge 1$,
\begin{equation}
\label{LagLimit}
      \frac{  ( \sqrt{\kappa n} )^{1/3} }{ ( \sqrt{\kappa} + \sqrt{n})^{4/3}}
          \Bigl( \lambda_{\max}(L_{\beta}) - (\sqrt{\kappa} + \sqrt{n})^2 \Bigr)  \Rightarrow TW_{\beta}.
\end{equation}
This covers all previous results for real/complex null Wishart matrices.
Comparing (\ref{HLimit}) and (\ref{LagLimit})  one sees that $O(n^{2/3} \eps)$ deviations in
the Hermite case should correspond to deviations of order 
$ (\kappa n)^{1/6} ( \sqrt{\kappa} + \sqrt{n} )^{2/3}  \eps $ $ = O ( \kappa^{1/2} n^{1/6} \eps)$ in the 
Laguerre case.
That is, one might expect bounds exactly of the form found in Theorem 1 with appearances of 
$n$ in each exponent replaced by $\kappa^{3/4} n^{1/4}$.  What we have is the following.

\begin{theorem}  
\label{Thm-Laguerre}
For all $\kappa + 1 > n \ge 1$, $0 < \eps \le 1 $ and $\beta \ge 1$$:$
$$
    \PP \Bigl(  \lambda_{\max}(L_{\beta})   \ge ( \sqrt{\kappa} + \sqrt{n})^2 (1 + \eps)  \Bigr) \le C  
                           e^{- \beta \sqrt{n \kappa} \eps^{3/2}  ( \frac{1}{\sqrt{\eps}}  \wedge   \left( \frac{\kappa}{ n} \right)^{1/4}    )/ C},
$$
and
$$
   \hspace{-.5cm}   \PP \Bigl(  \lambda_{\max}(L_{\beta})   \le  ( \sqrt{\kappa} + \sqrt{n})^2 (1- \eps) \Bigr) \le 
                           C^{\beta}  e^{-\beta  n\kappa   \eps^3  ( \frac{1}{\eps} \wedge \left( \frac{\kappa}{n} \right)^{1/2} )  / C}. 
$$
Again, $C$ is some numerical constant.
\end{theorem}

The right-tail inequality is extended to non-Gaussian matrices in \cite {FS}.
The rather cumbersome exponents in Theorem 2 do produce the anticipated decay,
though only for $\eps \le \sqrt{n/\kappa}$.
For $\eps \ge \sqrt{n/\kappa}$, the right and left-tails become linear and quadratic in $\eps$ respectively.
This is to say that the large deviation regime begins at the order $O(\sqrt{n/\kappa})$ rather than $O(1)$
as in the $\beta$-Hermite case. To understand this,
we recall that, normalized by $1/\kappa$, the counting measure of the $L_\beta$ points is asymptotically
supported on the interval with endpoints $(1 \pm \sqrt{n/\kappa})^2$.  This statement is precise with convergent
$n/\kappa$, and the limiting measure that of Mar\v{c}enko-Pastur.   Either way, $\sqrt{n/\kappa}$ is identified
as the spectral width, in contrast with the semi-circle law appearing in the $\beta$-Hermite case which is of width
one (after similar normalization).
Of course, in the more usual set-up when $c_1 n \le \kappa \le c_2 n$  ($c_1 \ge 1$ necessarily) all this is moot: the exponents above may then be replaced with $-\beta n \eps^{3/2}/C $ and $
-\beta n^2 \eps^3/C$ for $\eps$ in an $O(1)$ range with no loss of accuracy. And again, the large deviation tails were known in this setting 
for  $\beta = 1, 2 $.

An immediate consequence of the preceding  is a finite $n$ (and/or $\kappa$) bound on the variance of $\lambda_{\max}$ in line with the known limit theorems.  
This simple fact had only previously been available for GUE and LUE
(see the discussion in \cite{L2}).

\begin{corollary}
\label{Cor-VarianceBound}
Take $\beta \ge 1$.  Then,
\begin{equation}
 {\rm Var} \Bigl[ \lambda_{\max}(H_{\beta})  \Bigr]  \le C_{\beta}  \, n^{-1/3},  \  \  \  
   {\rm  Var} \Bigl[ \lambda_{\max}(L_{\beta} ) \Bigr] \le C_{\beta} \,  \kappa n^{-1/3}
\end{equation}
with now constant(s) $C_{\beta}$ dependent upon $\beta$.
\end{corollary}

The same computation behind Corollary \ref{Cor-VarianceBound}  implies that
$$\limsup_{n \rightarrow \infty} n^{p/6} \EE \big |\lambda_{\max}(H_{\beta}) - 2 \sqrt{n} \big  |^p <\infty$$ 
for any $p$, and similarly for $\lambda_{\max}(L_{\beta})$. Hence, we also conclude that 
all moments of the (scaled) maximal $H_{\beta}$ and $L_{\beta}$ eigenvalues converge
to those for the $TW_{\beta}$ laws (see \cite {BDMLMZ} for $\beta =2$).

Finally, there is the matter of whether any of the above upper bounds are tight.  
We answer this in the affirmative in the Hermite setting.

\begin{theorem}
\label{Thm-Lower}
There is a numerical constant $C$ so that
\begin{equation*}
     \PP \Bigl( \lambda_{\max}(H_{\beta})  \ge  2  \sqrt{n}   (1 + \eps) \Bigr)  \ge  C^{-\beta} e^{- C \beta n \eps^{3/2}},
\end{equation*}
and
\begin{equation*} 
     \PP \Bigl( \lambda_{\max}(H_{\beta})   \le   2  \sqrt{n}  (1 - \eps)  \Bigr)  \ge  C^{-\beta} e^{- C \beta n^2 \eps^{3}}.
\end{equation*}
The first inequality holds for all $n \ge 1, 0 < \eps \le 1,$  and $\beta \ge1$.  For the second inequality,
the range of $\eps$ must be kept
sufficiently small, $0 < \eps \le 1/C$ say.
\end{theorem}

Our proof of the right-tail lower bound takes advantage of a certain independence in the $\beta$-Hermite tridiagonals not immediately
shared
by the Laguerre models,  but the basic strategy also works in the Laguerre case.  Contrariwise, our
proof of the left-tail lower bound uses a fundamentally Gaussian argument that is not available in the Laguerre setting.

The next section introduces the tridiagonal matrix models and gives an indication of our approach.   The upper bounds
(Theorems 1 and 2, Corollary 3)
are proved in Section 3; the $H_{\beta}$ lower bounds in Section 4.  Section 5 considers the  analog of the right-tail upper
bound for the minimal eigenvalue in the $\beta$-Laguerre ensemble, this case holding the potential for some novelty granted
the existence of a different class of limit theorems (hard edge) depending on the limiting ratio $n/\kappa$.  
While our method does produce a bound, the conditions
on the various parameters are far from optimal. For this reason we relegate the statement, along with the proof and further discussion,
to a separate section.

\section{Tridiagonals}

The results of \cite{RRV} identify the general $\beta>0$ Tracy-Widom law through a random variational principle:
\begin{equation}
\label{TW}
  TW_{\beta} = \sup_{f \in L}  \left\{  \frac{2}{\sqrt{\beta}}   \int_0^{\infty} f^2(x) db(x)  - \int_0^{\infty} 
      \big [ (f'(x))^2  +  x f^2(x) \big ] dx   \right\}, 
\end{equation}
in which $x \mapsto b(x)$ is a standard Brownian motion and $L$ is the space of functions $f$ which vanish at the
origin and satisfy $\int_0^{\infty} f^2(x) dx = 1$, $\int_0^{\infty} [ (f'(x))^2 + x f^2(x)] dx < \infty$.   The equality here is in
law, or you may view (\ref{TW}) as the definition of $TW_{\beta}$.

This variational point of view also guides the proof of the convergence of the centered and scaled $\lambda_{\max}$
(of $H_{\beta}$ or $L_{\beta}$) to $TW_{\beta}$.  In particular, given the random tridiagonals which we are about to
introduce, one always has a characterization of $\lambda_{\max}$ through Raleigh-Ritz.  In \cite{RRV}, the point is
to show this ``discrete" variational problem goes over to the continuum problem (\ref{TW}) in a suitable sense. 
Furthermore, an analysis of the continuum problem has been shown to give sharp estimates on the tails of
the $\beta$ Tracy-Widom law (again see \cite{RRV}).  Our idea here is therefore retool those arguments for
the finite $n$, or discrete, setting.

We start with the Hermite case. 
Let $g_1, g_2, \dots g_n$ be independent Gaussians with
mean $0$ and variance 2.  Let also $\chi_{\beta},$
$\chi_{2\beta },$ $\dots,$ $ \chi_{(n-1)\beta} $ be independent
$\chi$ random variables of the indicated parameter.
Then, re-using notation, \cite{DE} proves that the $n$ eigenvalues of the random tridiagonal matrix
\begin{equation*}
\label{thematrix}
H_{\beta} =  \frac{1}{\sqrt{\beta}} \left[ \begin{array}{ccccc}  g_1 & \chi_{(n-1)\beta} &&& \\
\chi_{(n-1)\beta} &  g_2 & \chi_{(n-2)\beta} & &\\
&\ddots & \ddots & \ddots & \\
& & \chi_{2\beta } &  g_{n-1} & \chi_{\beta} \\
& & & \chi_{\beta} &  g_{n}  \\
 \end{array}\right]
\end{equation*}
have joint law (\ref{betadens}).\footnote{For $\beta=1$ or $2$ this can be seen by applying Householder
transformation to the ``full" GOE or GUE matrices, and appears to have been used first in a random matrix
theory context by Trotter \cite{Trotter}.}
Centering appropriately, we define: for $v = (v_1, \dots, v_n) \in \RR^n$, 
\begin{eqnarray}
      H(v) & =  & v^T  [ H_{\beta}  -  2  \sqrt{n} I_n ] v  \\
               & = & \frac{1}{\sqrt{\beta}}  \sum_{k=1}^n g_k v^2_k +  \frac{2}{\sqrt{\beta}} \sum_{k=1}^{n-1} \chi _{n-k} v_k v_{k+1}
          - 2 \sqrt n \sum_{k=1}^n v^2_k . \nonumber
\end{eqnarray}
The problem at hand (Theorem 1) then becomes that of estimating
\begin{equation}
\label{inH}
   \PP \Bigl( \sup_{||v||_2 = 1} H(v) \ge \sqrt{n} \eps \Bigr)  \  \   \mbox{   and   } \   \  \PP  \Bigl( \sup_{||v||_2 = 1} H(v) \le - \sqrt{n} \eps \Bigr), 
\end{equation}
where we have introduced the usual Euclidean norm $||v||_2^2 =  \sum_{k=1}^n v_k^2$.  To make the connection
between $H(v)$ and the continuum form  (\ref{TW}) even more plain we have the following.

\begin{lemma} For any $c > 0$ define
\label{Hc}
\begin{eqnarray}
\label{Hb} 
    H_c(v)  & =  &  \frac{1}{\sqrt{\beta}} \sum_{k=1}^n g_k v^2_k 
                  +  \frac{2}{\sqrt{\beta}} \sum_{k=1}^{n-1} \big [ \chi _{\beta(n-k)} - \EE (\chi _{\beta(n-k)}) \big] v_k v_{k+1}   \\
                   &  &  - c \sqrt n \sum_{k=0}^n (v_{k+1} - v_k)^2 
                         - {c \over \sqrt n} \sum_{k=1}^n k \, v^2_k  \nonumber
\end{eqnarray}
in which it is understood that $v_0 = v_{n+1} = 0$.  There exist numerical constants,  $a >  b > 0$,  so that
$$
        H_a(v) \le H(v) \le H_b(v)  \   \mbox{   for all v} \in \RR^n,
$$
granted $\beta \ge 1$.
\end{lemma}

We defer the proof until the end of the section, after a description of the allied $L_\beta$ set-up.  The point
of Lemma \ref{Hc} should be clear.  For instance, for an upper bound on the first probability in (\ref{inH}) one may replace
$H$ by $H_b$ with any sufficiently small $b >0$, and so on. 

The model for $L_{\beta}$ is as follows.
For $\kappa > n-1$, introduce the random bidiagonal matrix
\begin{equation*}
\label{Wmatrix} B_{\beta} =  \frac{1}{\sqrt{\beta}} \left[
\begin{array}{ccccc}
   {\chi}_{\beta \kappa} & &&& \\
\widetilde{\chi}_{\beta(n-1)} &   {\chi}_{\beta(\kappa -1)} &   & &\\
&\ddots & \ddots &  & \\
& & \widetilde{\chi}_{\beta 2} & {\chi}_{\beta( \kappa - n+2)} &    \\
& & & \widetilde{\chi}_{\beta} & {\chi}_{\beta( \kappa - n+1)}  \\
 \end{array}\right],
\end{equation*}
with the same definition for the $\chi$'s and again all variables independent.  (The
use of $\widetilde{\chi}$ is meant to emphasize this independence between the  diagonals.) 
Now \cite{DE} shows that it is the eigenvalues of $L_{\beta} =  (B_{\beta}) (B_{\beta})^{T}$ 
which have
the required joint density.\footnote{Once again, at $\beta=1,2$ this connection had
been noted previously (via Householder), see \cite{Silv} for example.} Note that $L_{\beta}$ does not
have independent entries.

Similar to before, we define
\begin{eqnarray*}
    \sqrt{\kappa}  L(v) & =  & v^T  \big ( L_{\beta}  -  (\sqrt{\kappa}+  \sqrt{n})^2 I_n  \big )  v  \\
               & = &   \frac{1}{\beta} \sum_{k=1}^n  \chi _{\beta(\kappa-k+1)}^2 v_k^2 
                 +  \frac{1}{\beta} \sum_{k=2}^n {\widetilde \chi}_{\beta(n -k+1)} ^2 v_k^2  \\
                 &  &          +  \frac{2}{\beta} \sum_{k=1}^{n-1} \chi _{\beta(\kappa-k+1)} {\widetilde \chi}_{\beta(n -k)} v_k v_{k+1}
                       - ( \sqrt \kappa + \sqrt n )^2 \sum_{k=1}^n v_k^2.              
\end{eqnarray*}
The added normalization by $\sqrt{\kappa}$ makes for better comparison with the Hermite case. With this, and
since $\kappa > n-1$, to prove Theorem 2 is to establish bounds on the following analogs of (\ref{inH}):
\begin{equation}
\label{inL}
   \PP \Bigl( \sup_{||v||_2 = 1} L(v) \ge \sqrt{n} \eps \Bigr)  \  \   \mbox{   and   } 
   \   \  \PP  \Bigl( \sup_{||v||_2 = 1} L(v) \le - \sqrt{n} \eps \Bigr)
\end{equation}
Finally, we state the Laguerre version of Lemma \ref{Hc}. (We prove only the latter as they are much the same).

\begin{lemma}
\label{Lc} For $c > 0$ set
\begin{eqnarray*}
  L_c(v) &  = &   \frac{1}{\sqrt{\beta}} \sum_{k=1}^n Z_k v_k^2   +  \frac{1}{\sqrt{\beta}}  \sum_{k=2}^n {\widetilde Z}_k v_k^2  
                            +  \frac{2}{\sqrt{\beta}} \sum_{k=1}^{n-1} Y_k v_k v_{k+1} \\
             & & - c \sqrt n \sum_{k=0}^n (v_{k+1}-v_k)^2  - {c \over \sqrt n} \sum_{k=1}^n k v_k^2, 
\end{eqnarray*}
where
\begin{eqnarray}
\label{LagNoise}
 & & Z_k = {1 \over \sqrt{\beta \kappa}  } \, \big ( \chi _{\beta(\kappa-k+1)}^2 -  \beta(\kappa-k+1) \big ),  \  \ 
  {\widetilde Z}_k = {1 \over \sqrt{\beta \kappa}  } \, \big ( {\widetilde \chi}^2_{\beta(n -k+1)} 
              - \beta(n -k+1) \big ),  \nonumber  \\
       &&   \mbox{    and   } Y_k =  {1 \over \sqrt{ \beta \kappa } } \big (   { \chi}_{\beta(\kappa  -k+1)} 
                     {\widetilde \chi}_{\beta(n-k)} 
                     - \EE [ {\chi} _{\beta(\kappa  -k+1)}   {\widetilde \chi}_{\beta(n-k)} ] \big ).  
\end{eqnarray}
Then, for all $\beta \ge 1$ there are constants $ a > b > 0$ so that $L_a(v) \le L(v) \le L_b(v)$ for all $v \in\RR^n$.
\end{lemma}

\begin{proof}[Proof of Lemma \ref{Hc}] 
Writing,
\begin{eqnarray*}
H(v) &=  & \frac{1}{\sqrt{\beta}} \sum_{k=1}^n g_k v^2_k 
                               + \frac{2}{\sqrt{\beta}} \sum_{k=1}^{n-1} \big ( \chi _{\beta(n-k)} 
                               - \EE [\chi_{\beta(n-k)}] \big ) v_k v_{k+1} \\
           & &          - \sum_{k=1}^{n-1} \EE [ \frac{\chi _{\beta(n-k)}}{\sqrt{\beta}} ] (v_{k+1} - v_k)^2  \\
                      & &  - \sum_{k=1}^{n-1} \big ( \sqrt n -  \EE [ \frac{\chi _{\beta(n-k)}}{\sqrt{\beta}} ] \big) (v_k^2+v_{k+1}^2)
                               - \sqrt n \,( v_1^2 +  v_n^2 ) 
\end{eqnarray*}
shows it is enough to compare, for every $v$,
$$ I(v) = \sqrt n \sum_{k=1}^{n-1} (v_{k+1} - v_k)^2  +  {1 \over \sqrt n}\sum_{k=1}^n k v^2_k $$
and
$$ 
 J(v) =    \sum_{k=1}^{n-1}    \EE [ \frac{\chi _{\beta(n-k)}}{\sqrt{\beta}}] (v_{k+1} - v_k)^2 
                       + \sum_{k=1}^{n-1} \big ( \sqrt n -   \EE [\frac{\chi_{\beta(n-k)}}{\sqrt{\beta}}] \big ) (v_k^2+v_{k+1}^2).
$$ 
For this, there is the formula
$$
     \EE \chi_r = 2^{1/2} {  \Gamma (r/2 + 1/2) \over \Gamma (r/2)} \, ,
$$
and, if $r \ge 1$, also the  bounds
$$
          \sqrt{r - 1/2 } \le \EE \chi_r \le  \sqrt{r}.
$$
(The upper bound is simply Jensen's inequality and holds for all $r >0$; the lower bound
requires more work.)  This translates to
$$
     {k \over 2 \sqrt n} \leq  \sqrt n -  \frac{1}{\sqrt{\beta}} \, \EE \chi_{\beta(n-k)}  \leq  {4 k \over \sqrt n} \, , 
$$
for $k \le n-1$, $\beta \ge 1$.
Hence, $J(v) \leq 8 I(v)$ for every $v$.
Conversely, if $k \leq n/2$, $ \EE [ \chi_{\beta(n-k)}/\sqrt{\beta}] \geq  \sqrt{n}/4 $
while if $k \geq n/2$, 
$ \sqrt n - \EE [ \chi_{\beta(n-k)}/\sqrt{\beta}] \geq  {\sqrt{n}}/4$. It follows that
$J(v) \geq  I(v)/16$. 
\end{proof}

\section{Upper Bounds} 

Theorems \ref{Thm-Hermite} and \ref{Thm-Laguerre} are proved, first for the 
$\beta$-Hermite case with all details present;  a second subsection explains the modifications
required for the $\beta$-Laguerre case.  The proof of Corollary \ref{Cor-VarianceBound} appears at the end.

\subsection{Hermite ensembles}

\paragraph{Right-tail.}  This is the more elaborate of the two.  The following is a streamlined version of what is needed.

\begin{proposition}  
\label{Prop-IntUp}
Consider the model quadratic form,
\begin{equation}
\label{modelform}
   H_b(v, z) =    \frac{1}{\sqrt{\beta}} \sum_{k=1}^n z_k v_k^2  - b \sqrt{n}  \sum_{k=0}^n (v_{k+1} - v_k)^2 - \frac{b}{\sqrt{n}}  \sum_{k=0}^n k v_k^2, 
\end{equation}
for fixed $b > 0$ and independent mean-zero random variables $\{z_k \}_{k =1, \dots, n}$ satisfying the uniform tail bound 
$ \EE [ e^{\lambda z_k} ]  \le e^{c\lambda^2}$ for all $\lambda \in \RR$ and some $c >0$.  There is a $C = C(b,c)$ so that
$$
       \PP \Big ( \sup_{||v||_2 = 1}  H_b(v, z)  \ge \eps \sqrt{n} \Big) 
       \le {(1- e^{- \beta /C})^{-1} } e^{ - \beta n \eps^{3/2}/C}
  $$
for all $\eps \in (0, 1]$ and $n \ge 1$.
\end{proposition}

The proof of the above hinges on the following version of integration by parts (as in fact
does the basic convergence result in \cite{RRV}). 

\begin{lemma}  
\label{parts}
Let $s_1, s_2, \ldots, s_k, \ldots$ be real numbers, and set
$S_k = \sum_{\ell = 1}^k s_\ell$, $S_0=0$. 
Let further $t_1, \ldots, t_n$ be real numbers, $t_0= t_{n+1} = 0$. Then,
for every integer $m\geq 1$,
$$ \sum_{k=1}^n s_k t_k = {1\over m} \sum_{k=1}^n   [S_{k+m-1} - S_{k-1} ]  t_k
                       +  \sum_{k=0}^n  \bigg ( {1\over m} \sum_{\ell =k}^{k+m-1} [S_\ell - S_k]
                             \bigg )   (t_{k+1}-t_k).$$
\end{lemma}

\begin{proof} For any $T_k$, $k = 0, 1, \ldots, n$, write
\begin{eqnarray*} 
 \sum_{k=1}^n s_k t_k
        & = &  \sum_{k=1}^n S_k (t_{k}-t_{k+1}) \\
        & =  & \sum_{k=0}^n [T_k - S_k] (t_{k+1}-t_k) - \sum_{k=0}^n T_k  (t_{k+1}-t_k) \\
        & =  & \sum_{k=0}^n [T_k - S_k] (t_{k+1}-t_k)  + \sum_{k=1}^n [T_k - T_{k-1}]  t_k . 
\end{eqnarray*}
Conclude by choosing $T_k =  {1 \over m} \sum_{\ell =k}^{k+m-1} S_\ell $, $k = 0, 1, \ldots, n$. 
\end{proof}

\begin{proof}[Proof of Proposition \ref{Prop-IntUp}]
Applying Lemma \ref{parts} with $s_k =  z_k$ and $t_k = v_k^2$ (bearing in mind that 
$v_0=v_{n+1} = 0$, and we are free to set $s_k =0$ 
 for $k \ge n+1$) yields
\begin{eqnarray*}
  \sum_{k=1}^n z_k v_k^2 
      & \leq  & \frac{1}{m } \sum_{k=1}^n | S_{k+m-1} - S_{k-1}| v_k^2 
           + \sum_{k=0}^n \Bigl( \frac{1}{m} \sum_{\ell = k}^{k+m-1} |S_\ell - S_k | \Bigr)
                            |v_{k+1}^2 - v_k^2|  \\
          & \leq &     \frac{1}{m}  \sum_{k=1}^n  \Delta_m(k-1) v_k^2
          + \sum_{k=0}^n \Delta_m(k)  |v_{k+1}+v_k| |v_{k+1} - v_k| 
\end{eqnarray*}
where 
\begin{equation}
    \Delta_m(k) = \max_{k + 1 \leq \ell \leq k+m} |S_\ell - S_k| ,  \  \mbox{   for   }   k = 0, \ldots, n.          
\end{equation}
Next, by the Cauchy-Schwarz inequality, for every $\lambda >0$,
$$  \frac{1}{\sqrt{\beta}} \sum_{k=1}^n z_k v_k^2 
     \leq   \frac{1}{m \sqrt{\beta}} \sum_{k=1}^n  \Delta_m(k-1) v_k^2
              + \lambda \sum_{k=0}^n (v_{k+1}-v_k)^2 
         + \frac{1}{4\lambda \beta} \sum_{k=0}^n \Delta_m(k)^2  (v_{k+1}+v_k)^2. $$
Choosing $\lambda = b \sqrt n$ we obtain
\begin{equation}
\label{bound1}
\sup_{||v||_2 =1} H_b(z, v)   \leq  \max_{1\leq k\leq n} \Bigl( \frac{1}{m \sqrt{\beta}} \, \Delta_m(k-1) 
    + \frac{1}{ 2 b \sqrt{n} \beta} \, \bigl[ \Delta_m(k-1)^2 + \Delta_m(k)^2 \bigr]
                -  b \frac{ k}{\sqrt n} \Bigr).
\end{equation}
And since whenever $(j-1)m + 1 \leq k\leq jm$, $1 \leq j \leq [n/m] +1$, it holds
$$ 
\Delta_m (k) \vee \Delta_m(k-1) \leq  2 \Delta_{2m} \big ( (j-1)m  \big),
$$
we may recast (\ref{bound1}) as in
\begin{eqnarray*}
\label{bound2}
\lefteqn{
\sup_{||v||_2 =1} H_b(z, v) } \\
& \leq &  \max_{1\leq j\leq [n/m] +1} \bigg ( {2\over m \sqrt{\beta} }\, \Delta_{2m}\big ( (j-1)m  \big) 
     + {4 \over b \sqrt{n} \beta} \,  \Delta_{2m}\big ( (j-1)m  \big)^2 
                - b \, {(j-1)m + 1 \over \sqrt n} \bigg ).  \nonumber                
\end{eqnarray*}

Continuing requires a tail bound on $\Delta_{2m}(J)$ for integer $J \ge 0$.  
By Doob's maximal inequality and our assumptions on $z_k$, for every $ \lambda> 0$ and $t  >0$,
$$  \PP \Bigl(  \max_{ 1 \leq \ell \leq 2m} S_\ell \geq t \Bigr)
         \leq  e^{- \lambda t} \, \EE \left[ e^{\lambda S_{2m}} \right]
       \leq  e^{- \lambda t +  2 c m \lambda ^2 } . 
  $$
Optimizing in $\lambda $, and then applying the same reasoning to the 
sequence $-S_\ell$ produces
$$ 
 \PP \Bigl( \max_{ 1 \leq \ell \leq 2m} | S_\ell | \geq t \Bigr) \leq  2 \, e^{-t^2/8cm}.
$$
Hence,
\begin{equation}
\label{bound3}
\PP \Bigl( \Delta_{2m} (J) \geq t \Bigr) \leq  2 \, e^{- t^2/8cm}, 
\end{equation}
for all integers $m \geq 1$ and $J \geq 0$, and every $t >0$.

>From (\ref{bound3}) it follows that
\begin{eqnarray}
\label{bd1}
 \lefteqn{
     \PP \Bigl( \max_{1\leq j\leq [n/m] +1 }  \Big ( \frac{2}{ m \sqrt{\beta} } \,  \Delta_{2m} \bigl( (j-1)m \bigr)
                     - \frac{[b (j-1)m +  1]}{2 \sqrt{n}} \Big ) \geq  \frac{ \eps \sqrt{n}}{  2}  \Bigr) }  \\
             & \le & \sum_{j=1}^{[n/m] +1}   
                       \PP \Bigl(  \frac{2}{m \sqrt{\beta}} \, \Delta_{2m} \bigl( (j-1)m  \bigr) \geq 
                       \frac{b [(j-1)m + 1]}{ 2\sqrt n}  + \frac{\eps \sqrt{n}}{2} \bigg)  \nonumber \\
          & \le  & 2 \sum_{j=1}^{[n/m] +1}  \exp \bigg ( - \frac{\beta m}{ 128 c} 
          \Big [  {b [(j-1)m + 1] \over \sqrt n}  + \varepsilon \sqrt n \Big ] ^2 \bigg),  \nonumber
\end{eqnarray}          
and similarly
\begin{eqnarray}
\label{bd2}
\lefteqn{ \hspace{-1cm}
\PP \Bigl(  \max_{1\le j \le [n/m] +1}   \Bigl( {4\over b \sqrt{n} \beta }\,  \Delta_{2m} \left( (j-1)m  \right)^2
                     -  \frac{b [(j-1)m + 1]}{2\sqrt n} \Bigr) \geq \frac{\varepsilon \sqrt{n}}{ 2} \Bigr) } \\
                     & \le   &   2 \sum_{j=1}^{[n/m] +1}  \exp \bigg ( - \frac{\beta b \sqrt{n}}{64 c m} 
          \Big[ \frac{b [(j-1)m + 1]}{\sqrt{n}}  + \eps \sqrt n \Big ]  \bigg ).   \nonumber                               
\end{eqnarray} 
Combined, this reads
\begin{eqnarray}
\label{bound4}
\lefteqn{
\PP \Bigl( \sup_{||v||_2 =1}  H_b(z, v) \ge  \eps \sqrt{n} \Bigr) }   \\
& \le &  \Bigl( \frac{2}{1- e^{-\beta \eps b m^2/ 64 c} } \Bigl)  e^{- \beta m n \eps^2/128 c}  +  \Bigl(  \frac{2}{ 1- e^{-\beta b^2/64 c}}  \Bigr)
e^{- \beta b \eps n /64 cm},  \nonumber
\end{eqnarray}
which we have recorded in full for later use.  In any case, the choice
$m = [ \eps^{-1/2}]$ will now produce the claim.
\end{proof}

We may now dispense of the proof of Theorem \ref{Thm-Hermite} (Right-Tail).
Before turning to the proof, we remark that if $\eps > 1$, one may run through the above argument
and simply choose $m=1$ at the end to produce the classical form of the large deviation inequality
(\ref{largedev}) known previously for $\beta =1,2$.  

We turn to the values $0 < \eps \leq 1$. The form (\ref{Hb}) is split into two pieces,
$$
                H_b(v) = H_{b/2}(v, g)   +  \tilde{H}_{b/2}(v, \chi), 
$$
Proposition \ref{Prop-IntUp}  applying to each. 

The first term on the right is precisely of the form
(\ref{modelform}) with each $z_k$  an independent mean-zero Gaussian of
variance $2$, which obviously  satisfies the tail assumption with $c=1$.  
The second term, $\tilde{H}_{b/2}(v, \chi)$, is a bit different, having noise present through the
quantity $\sum_{k=1}^{n-1} (\chi_{\beta(n-k)} - \EE \chi_{\beta(n-k)} ) v_k v_{k+1}$.  But carrying
out the integration by parts on $t_k = v_k v_{k+1}$ (and  $s_k = \chi_{\beta(n-k)} - \EE \chi_{\beta(n-k)}$),
will produce a bound identical to (\ref{bound2}), with an additional factor of $2$ before each appearance
of $\Delta_{2m}$.   Thus, we will be finished granted the following bound.

\begin{lemma}
\label{ChiProp}
For $\chi$  a $\chi$ random variable with parameter greater than or equal to one,
\begin{equation}
\label{chitail}
  \EE [ e^{\lambda \chi} ] \le e^{\lambda \EE \chi + \lambda^2 /2},  \mbox{   for all } \lambda \in \RR.
\end{equation}
\end{lemma}

\begin{proof}[Proof of Lemma \ref{ChiProp}]
This may be viewed as a consequence of the dimension free concentration 
inequalities for norms of Gaussian vectors \cite{L0}, but here is an elementary derivation. 

A $\chi$ of parameter $r$ has density function $f(x) = c_r x^{r-1} e^{-x^2/2}$ on $\RR_+$, requiring 
us to show  that
$$
   \frac{ \int_0^{\infty} x^p e^{-(x-\lambda)^2/2} dx }{ \int_0^{\infty} x^p e^{-x^2/2} dx } 
   \le \exp \bigg( \lambda \, \frac{ \int_0^{\infty} x^{p+1} e^{-x^2/2} dx }{ \int_0^{\infty} x^{p} e^{-x^2/2} dx} \bigg)
$$
for any $p \ge 0$.   The case $p=0$ can be done by hand, and we neglect it here. Also, we will consider 
only $\lambda > 0$, things being quite the same for $\lambda < 0$.

Taking logarithms and then differentiating in $\lambda$, we find that the above inequality (for $p >0$) is implied by
\begin{eqnarray}
\label{stochdom}
0 & \ge &  \int_0^{\infty} \int_0^{\infty} (x - \lambda - y)  (xy)^p e^{-(x-\lambda)^2/2}  e^{-y^2/2 } \, dx dy \\
    &  = &  {p}  \int_0^{\infty} \int_0^{\infty} (x^{p-1} y^p - x^p y^{p-1} ) e^{-(x-\lambda)^2/2}  e^{-y^2/2 } \, dx dy. \nonumber
\end{eqnarray}
As for this, let $X$ and $Y$ be positive random variables with density functions  $c_{q, \lambda} x^q e^{-(x-\lambda)^2/2} $ and
$c_{q} y^{q} e^{-y^2/2}$ respectively. 
Now $q=p-1 > -1$, 
and we still have $\lambda >0$. 
It is easy to convince oneself that $\EE Y \le \EE X$, which is exactly the second line of 
(\ref{stochdom}).
\end{proof}

\medskip

\paragraph{Left-Tail.}  This demonstrates yet another advantage of the  variational
picture afforded by the tridiagonal models.  Namely, the bound may be achieved by a suitable choice of 
test vector
since
$$
   \PP \Bigl( \sup_{||v||_2 = 1} H_a(v) \le  - 2 C  \sqrt{n}  \eps \Bigr)  \le \PP \Bigl( H_a(v)  \le  - 2 C  \sqrt{n} \eps 
      \Norm{2}{v}^2  \Bigr)
$$
for whatever  $\{v_k\}_{k=1,\dots,n}$ on the right hand side.  (We have thrown in a constant $C$ for reasons that
will be clear in a moment.)  Simplifying, we write
\begin{eqnarray}
\label{leftbound1}
\lefteqn{   \PP \Bigl( H_a(v)  \le  - 2 C  \sqrt{n}  \eps \Norm{2}{v}^2  \Bigr) } \\ 
& \le & 
   \PP \Bigl( H_a(v, g)  \le  -  C  \sqrt{n} \eps \Norm{2}{v}^2  \Bigr)  + \PP \Bigl(  \chi(v)  \le - C   \sqrt{n}  \eps \Norm{2}{v}^2  \Bigr),
\nonumber
\end{eqnarray}
where in $H_a(v,g)$ we borrow the notation of Proposition \ref{Prop-IntUp} and 
$$
    \chi(v) = \frac{2}{\sqrt{\beta}}  \sum_{k=1}^{n-1} \big (\chi_{\beta(n-k)} - \EE \chi_{\beta(n-k)} \big) v_k v_{k+1}.
$$

Focus on the first term on the right of (\ref{leftbound1}), and note that
\begin{eqnarray}
\label{leftbound2}
\lefteqn{   \PP \Bigl( {H}_a(v, g)  \le  -  C \sqrt{n} \eps \Norm{2}{v}^2  \Bigr) } \\
&=& \PP \bigg(   \Big ( \frac{2}{\beta} \sum _{k=1}^n v_k^4 \Big )^{1/2}  \mathfrak{g} \ge
                     C \sqrt n  \eps \sum _{k=1}^n v_k^2  - a \sqrt n \sum _{k=0}^n (v_{k+1} - v_k)^2
                         -  {a \over \sqrt n} \sum _{k=1}^n k v_k^2 \bigg) \nonumber
\end{eqnarray}
with $\mathfrak{g}$ a single standard Gaussian.  Our choice of $v$ is motivated as follows. The event in question
asks for a large eigenvalue (think of $ \sqrt{n}\eps$ as large for a moment) of an operator which mimics
negative Laplacian plus potential.  The easiest way to accomplish this would be for the potential to remain large
on a relatively long interval, with a flat eigenvector taking advantage.  We choose
\begin{equation}
\label{choice1}
      v_k =  \frac{k}{n \eps} \wedge \Big( 1 - \frac{k}{n \eps} \Big)  \mbox{   for   } k \le n \eps 
       \mbox{   and zero otherwise},
\end{equation}
for which
\begin{equation}
\label{appraise}
   \sum _{k=1}^n v_k^2 \sim 
        \sum _{k=1}^n v_k^4 \sim n\eps ,  \  \  \    
          \sum _{k=0}^n (v_{k+1} - v_k)^2 \sim {1 \over  n \eps },  \ \mbox{   and   }
        \sum _{k=1}^n k v_k^2 \sim   n^2 \eps^2.
\end{equation}
(Here $ a \sim b$ indicates that the ratio $a/b$ is bounded above and below by numerical constants.)
Substitution into (\ref{leftbound2}) produces, for choice of $C = C(a)$ large enough inside the probability
on the left,
$$
      \PP \Bigl( {H}_a(v, g)  \le  -  C \sqrt{n} \eps \Norm{2}{v}^2  \Bigr)  \le  e^{- \beta n^2 \eps^{3} / C} \, \,  
      \mbox{   for   }  \, \,  n \eps^{3/2} \ge 1.
$$
The restriction of the range of $\eps$ stems from the gradient-squared term; it also ensures that $\eps n \ge 1$
which is required for our test vector to be sensible in the first place.

 Next, as a consequence of Proposition \ref{ChiProp} (see (\ref{chitail})) we have the bound: for $c > 0$,
\begin{equation}
\label{chibound}
   \PP \Bigl(   \chi(v) \leq   -  c^2   \Norm{2}{v}^2 \Bigr)
          \leq \exp \bigg ( -  \beta  c \bigg (\sum _{k=1}^n v_k^2 \bigg)^2
                   \bigg/ 8 \sum _{k=1}^{n-1} v_k^2 v_{k+1}^2 \bigg).
\end{equation}
With $c = C  \sqrt{n} \eps$ and $v$ as in (\ref{choice1}), this
may be further bounded by $e^{-\beta n^2 \eps^3/C}$. Here too we should assume that $n \eps^{3/2} \ge 1$.

Introducing a multiplicative constant of the 
advertised form $C^{\beta}$ extends the above bounds to the full range of $\eps$
in the most obvious way. 
Replacing $\eps$ with $\eps/C$
throughout completes the proof.

\subsection{Laguerre ensembles}

\paragraph{Right-Tail.}  We wish to apply the same ideas from the Hermite case to the Laguerre
form $L_b(v)$ (for small $b$). Recall:
\begin{eqnarray}
\label{Lb}
    L_b(v) & =&  \frac{1}{\sqrt{\beta}} \sum_{k=1}^n Z_k v_k^2 + \frac{1}{\sqrt{\beta}} \sum_{k=2}^n \tilde{Z}_k v_k^2  +
                 \frac{2}{\sqrt{\beta}} \sum_{k=1}^{n-1} Y_k v_{k+1} v_{k}  \\
                 &  &  - b \sqrt{  n} \sum_{k=0}^n (v_{k+1}- v_k)^2 -  {b}\frac{1}{\sqrt{n}} \sum_{k=1}^n
                           k v_k^2. \nonumber
\end{eqnarray}
Here, $Z_k,  \widetilde{Z}_k$ and $Y_k$ are as defined in (\ref{LagNoise}),
and the appropriate versions of the tail conditions for these variables (in order to apply Proposition  \ref{Prop-IntUp}) 
are contained in the next two lemmas. 

\begin{lemma} For $\chi$ be a $\chi $ random variable of positive parameter,
$$ 
   \EE [ e^{\lambda \chi^2} ] \le  e^{  \EE[ \chi^2 ] (  \lambda   + 2  \lambda^2 ) }   
   \mbox{   for all real   } \lambda < 1/4.
$$
\end{lemma}

\begin{proof}
 With  $r   = \EE[ \chi^2 ] > 0$ and  $\lambda < {1\over 2}$,
$$ \EE [ e^{\lambda \chi^2} ] 
         = \Big ( {1 \over 1-2\lambda } \Big )^{r /2} . $$
Now, since $x \geq -{1\over 2}$ implies $\log (1+x) \geq x - x^2 $,  for
any $ \lambda \leq  {1\over 4}$ the right hand side of the above is less  
$  e^{ r (\lambda   + 2  \lambda ^2)} $
as claimed.      
\end{proof}

\begin{lemma}
\label{ProdProp}
 Let $\chi $ and ${\widetilde \chi} $ be independent $\chi  $
random variables, each of parameter larger than one. Then, for every $\lambda \in \RR$ such that $ | \lambda | < 1$, 
$$ \EE \Bigl[e^{\lambda (\chi  {\widetilde \chi} - \EE [ \chi {\widetilde \chi}] ) } \Bigr]
     \leq  {1 \over \sqrt { 1-\lambda ^2}} \, \exp \Big ( {\lambda ^2 \over 2(1-\lambda ^2)}
               \bigl[  \EE [\chi ]^2  + \EE [{\widetilde \chi}]^2  
                   + 2 \lambda  \EE [\chi ]  \EE [{\widetilde \chi}]  \bigr] \Big ) . $$
\end{lemma}

\begin{proof} For $|\lambda | < 1$, using inequality (\ref{chitail}) in the ${\widetilde \chi} $ variable,
$$  \EE [e^{\lambda \chi  {\widetilde \chi}} ]
         \leq \EE \big[ e^{\lambda \EE [ {\widetilde \chi}] \chi  + \lambda ^2 \chi ^2/2} \big] 
         = \int_{-\infty}^{\infty} \EE \big[ e^{ \lambda [\EE( {\widetilde \chi}) + s ] \chi  } \big] d \gamma (s)  $$
where $\gamma $ is the standard normal distribution on $\RR$. Now, for every $s$, with
(\ref{chitail}) in the $\chi $ variable,
$$  \EE \big[ e^{ \lambda (\EE ({\widetilde \chi}) + s ) \chi  } \big]
     \leq  e^{  \lambda (\EE [ {\widetilde \chi}] + s ) \EE [\chi] 
             + \lambda ^2 [\EE[{\widetilde \chi}] + s ]^2 / 2} .$$
The result follows by integration over $s$.
\end{proof}

What this means for the present application is that
\begin{equation}
\label{Zbound}
  \EE [ e^{\lambda Z_k } ], \, \,  \EE [ e^{\lambda \widetilde{Z}_k } ] 
     \le e^{2 \lambda^2} \mbox{   for all real } \lambda  \le \sqrt{\beta \kappa }/4,
\end{equation}
and
\begin{equation}
\label{Ybound}
  \EE[ e^{\lambda Y_k} ] \le 2 e^{12 \lambda^2} \mbox{   for all real } \lambda \mbox{   with   }
   |\lambda| \le \sqrt{\beta \kappa}/2\sqrt{2}.
\end{equation}
To proceed, we split $L_b(v)$ into three pieces now, isolating each of the noise 
components,  and focus  on
the bound for $\sup_{||v||_2 =1} L_{b/3}(v, Z)$ (the notation indicating (\ref{Lb}) with only the $Z$ noise term present).
One must take some care when arriving at the analog of (\ref{bound3}).  In obtaining an inequality of the
form $\PP( \Delta_m(J,Z) > t ) \le C e^{-t^2/C}$ we must be able to apply (\ref{Zbound})  (and (\ref{Ybound}) when considering 
the $Y$ noise term) with $\lambda = O(t/m)$.  But, examining (\ref{bd1}) and (\ref{bd2}) shows we only need consider $t$'s
of order $\sqrt{\beta} ( \sqrt{n} + \sqrt{\kappa} \eps) m$. Thus we easily get by via
\begin{eqnarray}
\label{bound4L}
\PP \Bigl( \sup_{||v||_2 =1}  L_b( v) \ge  \sqrt{\kappa} \eps \Bigr)  
& \le & C \, \PP \Bigl( \sup_{||v||_2 =1}  L_{b/3}( v, Z) \ge   \sqrt{\kappa} \eps /C \Bigr)   \\
& \le &  \Bigl( \frac{C}{1- e^{-\beta \eps \sqrt{\frac{\kappa}{n}}  m^2/C} } \Bigl)  e^{- \beta \kappa  m  \eps^2/C}  +  \Bigl(  \frac{C}{ 1- e^{-\beta / C}}  \Bigr)
e^{- \beta \sqrt{\kappa n}  \eps / Cm},  \nonumber
\end{eqnarray}
with $C = C(b)$, compare (\ref{bound4}).
Setting $m$ to be the nearest integer to $\frac{1}{\sqrt{\eps}} \left( \frac{n}{\kappa} \right)^{1/4}$ puts both exponential
factors on the same footing, namely on the order of $e^{- \beta \kappa^{3/4} n^{1/4} \eps^{3/2}/C}$, and
removes  all $\eps, \kappa,$ and $n$ dependence on the first prefactor.  Certainly the best decay possible,
but requires $\eps \le \sqrt{n/\kappa}$.   Otherwise, if $\eps \ge  \sqrt{n/\kappa}$, we simply choose 
$m=1$ in which case the second term of (\ref{bound4L}) is the larger and produces decay $e^{-\beta \sqrt{n\kappa} 
\eps/C}$.   Happily, both estimates agree at the common value $\eps =  \sqrt{n/\kappa}$.

\paragraph{Left-Tail.} It is enough to produce the bound for 
$\PP( L_a(v,Z) \le - C \sqrt{\kappa} \eps ||v||_2^2)$ for large $a$, given $v \in \RR^n$ and a $C = C(a)$ as in the Hermite
case.  Indeed,  (\ref{Zbound}) and (\ref{Ybound}) show that $L_a(v, \widetilde{Z})$ and $L_a(v, Y)$ will follow suit. 

We have the estimate
\begin{eqnarray}
\label{lagleft}
\lefteqn{
\PP \Bigl( L_a(v, Z) \le - C \sqrt{\kappa} \eps || v||_2^2 \Bigr) } \\
& = & \PP \Bigl( \sum_{k=1}^n (-Z_k) v_k^2 \ge \sqrt{\beta} \left[ C \sqrt{\kappa} \eps || v||_2^2 - a \sqrt{n} || \nabla v||_2^2 - (a/\sqrt{n}) || \sqrt{k} v ||_2^2 \right] \Bigr)  \nonumber \\
& \le & \exp\bigg ( - \beta  \, \frac{ [ C \sqrt{\kappa} \eps || v||_2^2 - a \sqrt{n} || \nabla v||_2^2 - (a/\sqrt{n}) || \sqrt{k} v ||_2^2]^2 }{ 8||v||_4^4} \bigg ).
\nonumber
\end{eqnarray}
Here we have introduced the shorthand
\begin{equation}
  || v||_4^4 = \sum_{k=1}^n v_k^4, \   || \nabla v ||_2^2 = \sum_{k=0}^n (v_{k+1} - v_{k})^2,  \   || \sqrt{k} v ||_2^2 = \sum_{k=1}^n k v_k^2, 
\label{shorthand}
\end{equation}
and  have also used the fact that (\ref{Zbound}) applies just as well to $-Z_k$.  In fact, the sign precludes any concern over the required
choice of $\lambda$. For the $Y$-noise term, care must be taken on this point, but one may check that all is fine given our selection of $v$
below.

For the small deviation regime, we use a slight modification of the Hermite test vector (\ref{choice1}), and set 
$$
v_k =   \Big ( \frac{\delta}{n \eps} k \Big) 
     \wedge \Big (1 - \frac{\delta}{n \eps} k \Big )  \mbox{    with    } \delta = (n/\kappa)^{1/2} 
$$
for $k \le n \eps/\delta$ and $v_k = 0$ otherwise. This requires $\eps \le \delta = (n/\kappa)^{1/2}$ in order to be sensible,
and produces the same appraisals for $|| v||_2^2, ||v||_4^4, ||\nabla||_2^2$, and $|| \sqrt{k} v||_2^2$
 as in (\ref{appraise}), with each appearance of $\eps$ 
replaced by $\eps/\delta$.  Substitution into (\ref{lagleft}) yields 
$$
 \PP \Bigl( L_a(v, Z) \le - C \sqrt{\kappa} \eps || v||_2^2 \Bigr) 
     \le \exp \bigg ( - (\beta/8) \kappa^{3/2} n^{1/2} \eps^3 
   \Bigl[C - O(1 \vee \frac{1}{\eps^3\kappa^{3/2} n^{1/2}} ) \Bigr]^2  \bigg).
$$
For $\eps > \sqrt{n/\kappa}$, notice that  the particularly simple choice of a constant $v$ gives
$$
   \PP \Bigl( L_a(1, Z) \le - C \sqrt{\kappa} \eps ||1||_2^2 \Bigr) 
   \le \exp \Bigl( - (\beta/8) \kappa  n \eps^2 [C - (2a/n) - a]^2 \Bigr).
$$  
Combined, these two bounds cover the claimed result, provided that $\kappa^{3/2} n^{1/2} \eps^3$
is chosen larger than one
in the former.  Extending this to the full range of $\eps$ and all remaining considerations are
the same as in the Hermite setting.

\subsection{Variances}

We provide details for  $\lambda_{\max}(H_{\beta})$, the Laguerre case is quite the same. (Neither
is difficult.)  Write
$$ 
   {\rm Var} \left[ \lambda_{\max}(H_{\beta})  \right]  
   \le   n \int_0^{\infty}  \PP \left( |\lambda_{\max}(H_{\beta}) - 2 \sqrt{n} | \ge  \sqrt{n} \eps \right)  d \eps^2, 
$$
and then split the integrand in two according whether  
$\lambda_{\max} \le 2 \sqrt{n}$  or $\lambda_{\max} > 2 \sqrt{n}$. 

First note that our upper bound on the probability that $\lambda_{\max}(H_{\beta}) - 2 \sqrt{n}  \le   -\sqrt{n} \eps$ 
applies to any $\eps = O(1)$.  Further, from the  tridiagonal model we see that $\lambda_{\max} $ 
stochastically dominates $(1/\sqrt{\beta}) \max_{1 \le k \le n}  g_k$. Hence, for $\delta < 0$ we have 
the cheap estimate
$\PP( \lambda_{\max}(H_{\beta}) \le -\delta \sqrt{n} ) \le e^{-\beta n^2 \delta^2}$, and 
thus
$$
  \PP \left( \lambda_{\max}(H_{\beta}) - 2 \sqrt{n}  \le   - \sqrt{n}  \eps \right) \le 
  C_{\beta} \, e^{-\beta n^2  ( \eps^{3}  \wedge  \eps^2 ) /C} 
$$
for all $\eps > 0$. This easily  produces
$$
 \left(   \int_0^{2}  +\int_2^{\infty} \right)  \PP \left( \lambda_{\max}(H_{\beta}) - 2 \sqrt{n}  \le -  \sqrt{n} \eps \right)  d \eps^2   \le C_{\beta} \,  n^{-4/3}.
$$ 

For the other range, recall that we mentioned at the end of proof for the right-tail upper bound that
the advertised estimate is easily extended to the large deviation regime (cf. (\ref{largedev})) to read
$$
     \PP \left( \lambda_{\max}(H_{\beta}) - 2 \sqrt{n}  \ge   \sqrt{n}  \eps \right) 
      \le C_{\beta}\,  e^{-\beta n ( \eps^{3/2} \vee \eps^2) /C}.  
$$
This results in
$$
  \left(  \int_0^2 + \int_2^{\infty}   \right)   \PP \left( \lambda_{\max}(H_{\beta}) - 2 \sqrt{n}  \ge \sqrt{n} \eps \right) d \eps^2   \le  C_{\beta} \,  n^{-4/3}
$$
and completes the proof.

\section{(Hermite) Lower Bounds}

\paragraph{Right-Tail.}  This follows from another appropriate choice of test vector $v$. 
To get started, write
\begin{eqnarray}
\label{Hlower}
 \PP \Bigl( \sup_{||v||_2 = 1}  H(v)  \ge  \sqrt n \eps \, \Bigr) 
      & \geq &
 \PP \Bigl( H_a(v) \geq   \sqrt n  \eps \,  ||v||_2^2  \Bigr)  \\
   &  \geq  &  \PP \Bigl( H_a (v, g) \geq  2  \sqrt{n} \eps  \,  ||v||_2^2  \Bigr) \, 
                \PP \Bigl( \chi  (v) <   \sqrt n \eps \, ||v||_2^2  \Bigr).  \nonumber
\end{eqnarray}
Here, as before, $\chi(v) = (2/\sqrt{\beta}) \sum_{k=1}^{n-1} ( \chi_{\beta(n-k)} -  \EE[ \chi_{\beta(n-k)}]) v_k v_{k+1}$.

Our choice of $v$ is arrived at by examining the first factor above: with, as in the left-tail upper bound, 
a standard Gaussian $\mathfrak{g}$,
\begin{eqnarray*}
\lefteqn{
   \PP \Bigl( H_a (v, g) \geq  2 \varepsilon  \sqrt n \,  ||v||_2^2  \Bigr) } \\
   & = & \PP \bigg(  \Big ( \frac{2}{\beta} \sum_{k=1}^n v_k^2 \Big)^{1/2} \mathfrak{g} 
        \ge 2 \sqrt{n} \eps \sum_{k=1}^n v_k^2 + a \sqrt{n} \sum_{k=0}^n (v_{k+1} - v_k)^2 
            + \frac{a}{\sqrt{n}} \sum_{k=1}^n k v_k^2 \bigg).
\end{eqnarray*}
 Now the intuition
is that the eigenvalue (of a discretized $-d^2/dx^2 +$ potential) is being forced large positive, so the potential
should localize with the eigenvector following suit.  

Let then
$$
   v_k = \sqrt{\eps} k \wedge \big ( 1- \sqrt{\eps} k \big )  \  \ \mbox {   for   } k \le \eps^{-1/2}
    \mbox{   and otherwise  } 0,
$$
where we will assume that $n \ge \eps^{-3/2} \ge \eps^{-1/2}$.  With these choices 
we have 
$$
   || v ||_2^2 \sim || v||_4^4 \sim \frac{1}{\sqrt{\eps}}, \  || \nabla v ||_2^2 \sim \sqrt{\eps}, \  || \sqrt{k} v ||_2^2 \sim \frac{1}{\eps},
$$   
(recall the notation from (\ref{shorthand}))
and thus the existence of a constant $C = C(a)$ so that
$$
   \PP \Bigl( H_a (v, g) \geq  2 \varepsilon  \sqrt n \,  ||v||_2^2  \Bigr) \ge \frac{1}{C} \, e^{-C \beta n \eps^{3/2}}.
$$
Similarly, returning to the second factor on the right hand side of (\ref{Hlower}) and  invoking the estimate 
(\ref{chibound}) we also have 
$$
  \PP \Bigl(  \chi(v) \ge  \sqrt{n} \eps ||v||_2^2 \Bigr) \le e^{ - \beta n \eps^{3/2} /C}
$$
for the same choice of $v$.  And granted $n \eps^{3/2} \ge 1$, it follows
that $ \PP (  \chi(v) <  \sqrt{n} \eps ||v||_2^2) \ge  1- e^{ - 1/C}$
throughout this regime. That is, 
$$
    \PP \Bigl( \sup_{||v||_2 = 1}  H(v)  \ge  \sqrt n \eps \, \Bigr)  
    \ge \frac{1}{C} \, e^{-C \beta n \eps^{3/2}} \mbox{   whenever   } n \eps^{3/2} \ge 1.
$$
When $n \eps^{3/2} \le 1$, write
$$
   \PP \Bigl( \sup_{||v||_2 = 1}  H(v)  \ge  \sqrt n \eps \, \Bigr) \ge  \PP \Bigl( \sup_{||v||_2 = 1}  H(v)  \ge  \sqrt n \eps_0 \, \Bigr) \ge  \frac{1}{C e^{\beta C}} \ge \frac{1}{C e^{\beta C}} \,
   e^{- C \beta n \eps^{3/2}}, 
$$
where $\eps_0 = n^{-2/3} \le 1$ to produce the advertised form of the bound for all $n$ and $\eps$.

\medskip

\paragraph{Left-Tail.}  This relies heavily on the right-tail upper bound.  
The first step is to reduce to a Gaussian setting via independence: for whatever $b> 0$,
$$
  \PP \Bigl( \sup_{||v||_2 = 1}  H_{2b}(v) \le - \sqrt{n} \eps \Bigr) 
  \ge   \PP \Bigl( \sup_{||v||_2 = 1}  {H}_{b}(v, g) \le - 2 \sqrt{n} \eps \Bigr) \, 
   \PP \Bigl( \sup_{||v||_2 = 1}  \widetilde{H}_{b}(v, \chi) \le  \sqrt{n} \eps \Bigr).
$$
Here we also use the notation of the proof of Theorem \ref{Thm-Hermite} (right-tail), from which we know that
$$
    \PP \Bigl( \sup_{||v||_2 = 1}  \widetilde{H}_{b}(v, \chi) \ge  \sqrt{n} \eps \Bigr) \le C e^{-n \eps^{3/2} /C}.
$$
(As $\beta \ge 1$ we are simply dropping it from the exponent on the right at this stage.)
Hence, if as we regularly have start with an assumption like $n \eps^{3/2} \ge C^2 \ge 1$, it follows that
$$
    \PP \Bigl( \sup_{||v||_2 = 1}  H_{2b}(v) \le - \sqrt{n} \eps \Bigr) \ge (1 -e^{-1})  \PP \Bigl( \sup_{||v||_2 = 1}  {H}_{b}(v, g) \le - 2 \sqrt{n} \eps \Bigr). 
$$ 
Turning  to $H_b(v, g)$ we make yet another decomposition of the noise term.  Let $L$ be an integer ($1\le L\le n$) to be specified.  Set
$S_L = \frac{1}{L} \sum_{k=1}^L g_k$, and
$$  
   \eta_k = g_k - \frac{1}{L} S_L \mbox{   for   } 1\le k \le L, \, \eta_k = g_k \mbox{   for   }    L < k \le n.
$$ 
Note that the family $\{ \eta_k \}_{k=1, \dots, n}$ is independent of $S_L$.  If the procedure of Proposition \ref{Prop-IntUp} could be applied
to $H_b(v, \eta)$, we would have an event of probability larger than $1-C e^{-n\eps^{3/2} /C}$ (again we simply drop the beta dependence
at this intermediate stage)
on which
\begin{equation}
\label{etabound}
   \frac{1}{\sqrt{\beta}} \sum_{k=1}^n \eta_k v_k^2 - b \sqrt{n} \sum_{k=1}^n (v_{k+1}-v_k)^2 - \frac{b}{2\sqrt{n}} \sum_{k=1}^n k v_k^2 \le \sqrt{n} \eps \sum_{k=1}^n v_k^2. 
\end{equation}
Since we are still working under the condition $n \eps^{3/2} \ge C^2$, this is to say that there is an event of probability a least $1- 1/e$, 
depending only of the $\eta_k$'s, and  on which
$$
  H_b(v, g) \le \frac{1}{\sqrt{\beta}} \,   S_L \sum_{k=1}^L v_k^2 - \frac{b}{2 \sqrt{n}} \sum_{k=1}^n k v_k^2 + \sqrt{n}{\eps} \sum_{k=1}^n v_k^2,
$$
for every $v \in \RR^n$.  If we now choose $L+1\ge 6 n \eps/b$, we  have further
$$
    H_b(v, g) \le \frac{1}{\sqrt{\beta}} S_L \sum_{k=1}^L v_k^2 + \sqrt{n} \, \eps \sum_{k=1}^L  v_k^2 - 2 \sqrt{n} \, {\eps} \sum_{k=L+1}^n v_k^2
$$
on that same event.  Note this choice requires $\eps \le b/6$;  it is here that the range of valid epsilon gets cut down in our final statement.
In any case, putting the last remarks together we have proved that
$$
  \PP \Bigl( \sup_{||v||_2 = 1}  {H}_{b}(v, g) \le - 2 \sqrt{n} \, \eps \Bigr) \ge (1 - e^{-1}) 
  \PP \left( S_L \le - 3  \sqrt{n \beta} \, \eps \right)
$$
and so also
$$  
    \PP \Bigl( \sup_{||v||_2 = 1}  {H_{2b}}(v) \le - \sqrt{n} \, \eps \Bigr)   \ge \frac{1}{C} \, e^{- C \beta n^2 \eps^3},
$$
again
under the constrains $n \eps^{3/2} \ge C^2$ and $\eps \le b/6$.  The last inequality follows as $S_L$ is a mean-zero Gaussian with variance of order $(n\eps)^{-1}$.

The range $n \eps^{3/2} \le C^2$ is handled as before, 
$$
   \PP \Bigl( \sup_{||v||_2 = 1 } H_{2b}(v) \le - \sqrt{n} \, \eps \Bigr)    \ge    \PP \Bigl(   \sup_{||v||_2 = 1 } H_{2b}(v) \le - \sqrt{n} \, \eps_0 \Bigr)   \ge \frac{1}{Ce^{\beta C^5}} 
       \ge  \frac{1}{Ce^{\beta C^5}} \, e^{-C \beta n^2 \eps^3},     
$$
where $\eps_0 = (C^2/n)^{2/3}$.  As $\eps_0$ must lie under $b/6$, this last selection requires $n \ge (6/b)^{3/2} C^2$, 
but smaller values of $n$ can now be covered by adjusting the constant.

It remains to go back and verify that
$\PP ( \sup_{||v||_2 = 1} H_b(v, \eta) \ge \sqrt{n} \eps ) \le Ce^{-n \eps^{3/2}/C}$. 
The only reason that Proposition \ref{Prop-IntUp} cannot be followed verbatim is that the $\eta_k$'s
are not independent, the first $L$ of them being tied together through $S_L$.  We need the appropriate Gaussian tail inequality for 
the variables
$$
  \triangle_m(k, \eta) = \max_{k <  \ell \le k+m}  \bigg |  \sum_{j=k}^{\ell} \eta_j \bigg |,
$$
and, comparing with (\ref{bound3}), shows that an estimate of type $\PP (  \triangle_m(k, \eta) > t) \le C e^{-t^2/Cm}$   suffices.  But
$$
     \sum_{j=k}^{\ell} \eta_j  =  \sum_{j=k}^{\ell} g_j  + (\ell \wedge L - k \wedge L) S_L,
$$
and so
$$
 \PP \big ( \triangle_m(k, \eta) > t \big ) \le \PP \big ( \triangle_m(k, g) > t/2 \big ) + \PP (m S_L > t/2 ). 
$$
The first term we have already seen to be of the required order, and the second is less than $e^{- L t^2/8m^2}$.  
Since we  only apply this 
bound in the present setting when $L = Cn \eps \ge C \eps^{-1/2}$ and  $m = [\eps^{-1/2}]$
(the choice made in Proposition \ref{Prop-IntUp}), we have that
  $\PP (m S_L > t/2 ) \le  e^{-t^2/Cm}$,  and the proof is complete.

\section{Minimal Laguerre Eigenvalue}

While not detailed there, the results of \cite{RRV} will imply that
\begin{equation}
\label{lambdamin}
   \frac{ (\sqrt{\kappa n})^{1/3}}{ (\sqrt{\kappa} - \sqrt{n})^{4/3}} 
    \Bigl( ( \sqrt{\kappa} - \sqrt{n})^2 - \lambda_{\min}(L_{\beta}) \Bigr) \Rightarrow TW_{\beta},
\end{equation}
whenever $\kappa, n \rightarrow \infty, \kappa/n \rightarrow c > 1$.  This appraisal was long understood
for the minimal eigenvalue of L$\{$O/U$\}$E, and has recently been extended to non-Gaussian 
versions of those ensembles in \cite{FS}.
The condition $\kappa/n \rightarrow c > 1$ keeps the limiting spectral density supported
away from the origin, resulting in the same soft-edge behavior that one has for $\lambda_{\max}$.
If instead
$\kappa - n $ remains fixed in the limit, one has a different scaling and different limit law(s) for $\lambda_{\min}$,
the so-called hard-edge distributions.  Granted the existence of the ``hard-to-soft transition"
for all $\beta >0$  (see \cite{BF} and 
\cite{RR}) it is believed that (\ref{lambdamin}) holds as long as
$\kappa - n \rightarrow \infty$, but (to the best of our
knowledge)  this has not been explicitly worked out in any setting.

We only consider the  analogue of the right-tail upper bound for $\lambda_{\min}$
and have  the following.

\begin{theorem} 
\label{lastthm}
 Let $\beta \ge 1$ and  $\kappa \ge n + {1}$. Then,
\begin{equation}
\label{minbound}
 \PP \Bigl( \lambda_{min} (L_{\beta})  \le  ( \sqrt{\kappa} - \sqrt{n})^2 (1-\eps) \Bigr) 
  \le C e^{-\beta (\kappa n)^{1/4} (  \sqrt{\kappa} - \sqrt{n}) \eps^{3/2} / C},
\end{equation}
for a numerical constant $C$ and
all $0 < \eps \le   \sqrt{\frac{n}{\kappa}} ( \alpha^{14} \wedge   \alpha^{2}n^{-2/5} )$ 
in which $\alpha = 1 - \sqrt{n/\kappa}$.
\end{theorem} 

According to (\ref{lambdamin}), the deviations are of the order of 
$(\sqrt {\kappa n})^{1/3} (\sqrt \kappa  - \sqrt n)^{2/3} \eps$, which
explains the exponent in (\ref {minbound}).
Our condition on $\eps$ is certainly not very satisfactory, although
still sensible to the fluctuations in (\ref{lambdamin}). One would hope for
the  range of $\eps$ to be 
understandable in terms of the soft/hard edge picture $-$ what we have here arises
from technicalities.   On the other hand, if we 
place an additional, ``soft-edge" type, restriction on $\kappa$ and $n$, we 
obtain a more natural looking estimate.

\begin{corollary}
\label{Easy}
  Again take $\beta \ge 1$, but now assume that $\kappa > c n$ for $c >1$.
The right hand side of (\ref{minbound}) may then be replaced by
$C e^{-\beta n \eps^{3/2}/C}$ for a $C = C(c)$, with the resulting bound valid for all $0 < \eps \le 1$.
\end{corollary}

The last statement should be compared with Corollary V.2.1(b) of  \cite{FS}, which applies
to classes of non-Gaussian matrices.

As to the proof, we proceed in a by now familiar way. We first set
$$
   \sqrt{\kappa}  L(v) = v^T \Bigl(  (\sqrt{\kappa} - \sqrt{n})^2 - L_{\beta}  \Bigr) v.
$$
Then, after a rescaling of $\eps$,  we will prove the equivalent  
$$
\PP \Bigl( \sup_{||v|| = 1} L(v)  \ge  \alpha^{4/3} \sqrt{n} \eps \Bigr) \le C e^{-\beta n \eps^{3/2} / C}
\mbox{   for   } \eps \le   \min ( \alpha^{44/3} , \alpha^{8/3}n^{-2/5} ).
$$ 
Similar to the strategy employed above, a series of algebraic manipulations
shows that we can work instead with the simplified quadratic form
\begin{eqnarray}
\label{L'}
   L'(v)  & = &   \frac{1}{\sqrt{\beta}} \sum_{k=1}^n (-Z_k) v_k^2   + \frac{1}{\sqrt{\beta}} \sum_{k=2}^n (- {\widetilde Z}_k) v_k^2  
                            + \frac{2}{\sqrt{\beta}} \sum_{k=1}^{n-1} (- Y_k ) v_k v_{k+1}   \\
              &  & -   \sum_{k=1}^{n-1}   {1 \over \beta  \sqrt{ \kappa} } \,
                     \EE [ \chi _{\beta(\kappa -k+1) } {\widetilde \chi} _{ \beta(n-k)} ] (v_{k+1} + v_k)^2  
                 - { \alpha ^2\over \sqrt n} \sum_{k=1}^n k v_k^2 . \nonumber
\end{eqnarray}
(The condition $\kappa \ge n + {1}$ in Theorem \ref{lastthm} is used in passing from $L$ to $L'$.)  

We remark that
under the added condition $\kappa > c n$  for  $c >1$,  $\alpha$ is bounded uniformly from below and 
${1 \over \beta \sqrt{ \kappa} } \, \EE [ \chi _{\beta(\kappa-k+1)} {\widetilde \chi} _{\beta(n -k)} ] $ is bounded
below by a constant multiple of  $ \sqrt{n-k}$.
Hence, the deterministic part of $L'$ is bounded above by a small negative multiple  of
$  \sqrt{n} \sum_{k=1}^{n-1}  (v_{k+1} + v_k)^2 +   \frac{1}{\sqrt{n}} \sum_{k=1}^n k v_k^2$.      The proof of Corollary  \ref{Easy}         
is then identical to that of the right-tail upper bound for $\lambda_{\max}(L_{\beta})$.

Back to Theorem \ref{lastthm} and $\alpha$'s unbounded from below, we begin by rewriting the noise term in $L'$ as
$\frac{1}{\sqrt{\beta}}$ times
\begin{eqnarray*}
\lefteqn{  \sum_{k=1}^n  (-Z_k)  v_k^2   + \sum_{k=2}^n (- {\widetilde Z}_k) v_k^2  
                            +2 \sum_{k=1}^{n-1}  (-Y_k) v_k v_{k+1} }  \\
              &= & \sum_{k=1}^n (-  U_k ) v_k^2 
                      + \sum_{k=2}^n (- {\widetilde Z}_k) ( v_k^2 - v_{k-1}^2)
                   + \sum_{k=1}^{n-1} (-Y_k)v_k (v_{k+1} + v_k),  
\end{eqnarray*}
in which
$$ U_k = {1 \over \sqrt{\beta \kappa} } \Bigl[  ( \chi _{\beta(\kappa-k+1)} - {\widetilde \chi}_{\beta(n -k)} )^2
     - \EE \bigl[ ( \chi _{\beta(\kappa-k+1)} - {\widetilde \chi}_{\beta(n -k)} )^2 \bigr] \Bigr] ,
    \quad k = 1, \ldots , n ,
$$      
(with the convention that ${\widetilde \chi}_0 = 0$).          
The idea is the following.  For moderate $k$, $\mbox{Var}(U_k) = O(\alpha^2)$, and thus it is this
contribution to the noise which balances the drift term $\frac{\alpha^2}{\sqrt{n}} \sum_{k=1}^n k v_k^2$.
Also, one may check that in the continuum limit the optimal $v$ is such that $|v_k + v_{k+1}| = o(1)$, and so the $\widetilde{Z}$ and $Y$
terms should ``wash out".

We complete the argument in two steps.  In step one, we simply drop the $\widetilde{Z}$ and $Y$ terms and apply the
method in Proposition 7 to the further simplified form
\begin{equation}
\label{LU}
 L(v,U) = \frac{1}{\sqrt{\beta}} \sum_{k=1}^n (- {U}_k) v_k^2  
                          -   \sum_{k=1}^{n-1}   
                      \frac{\EE [ \chi _{\beta(\kappa -k+1) } {\widetilde \chi} _{ \beta(n-k)} ] }{ \beta \sqrt{\kappa}} (v_{k+1} + v_k)^2  
                 - { \alpha ^2\over \sqrt n} \sum_{k=1}^n k v_k^2. 
\end{equation}
Even here we loose a fair bit in our estimates (resulting in non-optimal on $\eps$) due 
to the variable coefficient in the energy term.
Step two shows that, under yet additional restrictions on $\eps$, the $\widetilde{Z}$ and $Y$ noise
terms may be absorbed into $L(v, U)$.

\medskip

{\em Step 1.}  We wish to prove $\PP( \sup_{||v||=1} L(v, U) \ge \alpha^{4/3} \sqrt{n} \eps) \le C e^{- \beta n \eps^{3/2}/C}$ for some
range of $\eps > 0$. (The optimal range being $0 < \eps  \le (\kappa/n)^{1/2} \alpha^{2/3}.)$  A first ingredient is a tail bound  on the $U_k$ variables, for which we first bring in the following.

\begin{lemma}(Aida, Masuda, Shigekawa \cite{AMS})
Given a measure $\eta$ on the line which satisfies a logarithmic Sobolev inequality with constant $C > 0$,
there is the estimate
$$
   \int e^{\lambda (F^2 - \EE [F^2])} d \eta \le 2 \, e^{8 C \lambda^2 \EE[ F]^2}
    \mbox{   whenever   }   | \lambda | \le \frac{1}{16 C},
$$
for any  $1$-Lipschitz function $F$.
\end{lemma}

As a consequence, we have that:

\begin{corollary}
Let $\chi$ and $\widetilde{\chi}$ be independent $\chi$ random variables (each of parameter
larger than one) and set $U = (\chi - \widetilde{\chi})^2$ and $\sigma = \EE [ \chi - \widetilde{\chi}]$.
There exists a numerical constant $C>0$ such that 
$$
   \EE [ e^{\lambda ( U - \EE U)} ] \le C e^{C \sigma^2 \lambda^2  }
$$
for all real $\lambda \in (-1/C, 1/C)$.
\end{corollary}

Indeed, by the general theory  (see Thm. 5.2 of \cite{L0} for example) the
distribution of the pair $( \chi, \widetilde{\chi})$ on $\RR_+ \times \RR_+$ satisfies a logarithmic 
Sobolev inequality.  The lemma then applies with  $F(x,y) = x-y$.  In our setting, we record
this bound as
$$
\label{varvar}
   \EE [ e^{\lambda U_k} ] \le C e^{ C  \sigma_k^2  \lambda^2}  \mbox{   for   }
    |\lambda| < \sqrt{\beta \kappa}/{C} \mbox{    and    } \sigma_k^2 = \EE[ U_k]^2.
$$

Picking up the thread of Proposition 7, the variable coefficient in the energy term of $L(v, U)$ is dealt with by applying the Cauchy-Schwarz argument
with $\lambda = \lambda_k$ defined by
\begin{equation}
\label{varlam}
  \lambda_k =  \frac{\EE [ \chi _{\beta(\kappa -k+1) } {\widetilde \chi} _{ \beta(n-k)} ] }{ \beta \sqrt{\kappa}}, \   k = 1, \dots, n-1,
\end{equation}
compare (3.3).  Schematically, we are left to bound 
\begin{equation}
\label{lastsum}
   \sum_{j=1}^{[n/m]} \PP \Bigl(  \frac{1}{m \sqrt{\beta}} \Delta_{m} (jm, U)  \vee \frac{1}{\lambda_{jm} \beta} \Delta_{m} (jm, U)^2  \ge \frac{\alpha^2}{\sqrt{n}} jm + \eps \alpha^{4/3} \sqrt{n} \Bigr)
\end{equation}
for our choice of integer $m$.  The $\Delta_m(\cdot, U)$ notation stands in analogy to that used in Section 3.  Note we have taken the liberty 
to drop various constants and shifts of indices in the above display (which are irrelevant to the upshot).

Here the dependence of $\sigma_k$ and $\lambda_k$  on the relationship between $n$ and $\kappa$ comes into play. While at the
top of the form everything works as anticipated, these quantities behave unfavorably for $k$ near $n$.
For this reason we deal with the sum (\ref{lastsum}) by dividing the range into $j \le n/2m$ and $j > n/2m$ with
the help of the appraisals:
\begin{equation}
\label{siglam}
  \sigma_k^2 \le  \left\{ \begin{array}{ll} C \alpha^2, &  1 \le k \le n/2, \\ C, & n/2 < k \le n. \end{array} \right.  \   \   \
  \lambda_k \ge \left\{ \begin{array}{ll} \sqrt{n}/C, &  1 \le k \le n/2, \\  \frac{\alpha}{C} \sqrt{n-k}, & n/2  < k < n. \end{array} \right.
\end{equation}   

Restricted to $j \le n/2m$ (and hence substituting $\sigma^2_{jm}  = C \alpha^2$, $\lambda_{jm} = \sqrt{n}/C$), 
the sum (\ref{lastsum}) can be 
bounded by $C e^{- \beta n \eps^{3/2}/C}$ upon choosing $m = [ \eps^{-1/2} \alpha^{-2/3}]$.
This  holds for all values of $\eps$ so long as the choice of $m$ is sensible,
requiring that $\eps \ge \alpha^{-4/3} n^{-2}$. But this is ensured if $\kappa \ge n+1$ and $\eps^{3/2} n \ge 1$ 
(the former having been built into the hypotheses and the latter we may always assume). 

On the range $j \ge n/2m$ the $\eps$ term on the right hand side within the probabilities is of no help, 
and we use, along with $ \sigma_{jm}^2 \le C$ and $\lambda_{jm} \le \alpha \sqrt{n-jm}/C$,  the crude estimates
\begin{eqnarray*}
  \sum_{n/2m \le j < n/m} \PP \Big ( \Delta_m(jm, U) \ge \sqrt{\beta}  \alpha^2 m^2 \frac{j}{\sqrt{n}} \Big )
  &  \le &  C \sum_{j \ge n/2m} e^{-\beta m^3 \alpha^4 j^2/ C n} \\
  & \le  & {C}{m^{-2} \alpha^{-4}} \, e^{-\beta m \alpha^4 n/C},
\end{eqnarray*}
and
\begin{eqnarray*}
 \sum_{n/2m \le j < n/m} \PP \Big ( \Delta_m(jm, U)^2 \ge  
 {\beta} \alpha^2  m \lambda_{jm}  \frac{j}{\sqrt{n}} \Big)
 & \le & C \sum_{1 \le j \le n/2m}  e^{ - \beta \alpha^3 (n/m)^{1/2} j^{1/2} /C}  \\
 &  \le  &  C \big ( 1 +  ( m/ n {\alpha}^{6} ) \big )  \,  e^{-\beta \alpha^3 (n/m)^{1/2}/C}.
\end{eqnarray*}
The choice of $m = [ \eps^{-1/2} \alpha^{-2/3}]$ being fixed, we can bound each of the above by the 
desired 
$C e^{-\beta n \eps^{3/2}/C}$ only by restricting 
$\eps$ to be sufficiently small.  The first estimate requires $\eps \le \alpha^{20/3}$, the second requires
in addition that $\eps \le \alpha^{8/3}n^{-2/5}$ (and again uses $n \eps^{3/2} \ge 1$).

In summary 
\begin{equation}
\label{almost}
\PP \Big( \sup_{||v||=1} L(v, U) \ge \alpha^{4/3} \sqrt{n} \eps \Big) 
\le C e^{- \beta n \eps^{3/2}/C}  \, \, \mbox{   if   } \, \,  0 < \eps \le \min( \alpha^{20/3} , \alpha^{8/3}n^{-2/5}).
\end{equation}
It is perhaps worth mentioning here that the bounds on $\lambda_k$ and $\sigma_k^2$ for the range $k \ge n/2$ introduced in (\ref{siglam})
may be improved slightly, though not apparently with great effect on the final result.

\medskip

{\em Step 2.}  To absorb the $\widetilde{Z}, Y$ noise terms, we show that
$
  L'(v) \le \tilde{L}(v,U) + \mathcal{E}(\widetilde{Z}, Y, v)
$
with a new form $\tilde{L}(v, U)$ comparable to $L(v,U)$, and an ``error" term $\mathcal{E}$ for
which we have $\PP(\mathcal{E} \ge  \alpha^{4/3} \sqrt{n} \eps ) \le C e^{- \beta n \eps^{3/2}/C}$, at
least for some range of $\eps>0$.  What follows could almost certainly be improved upon.

Define, for $k = 1, \dots, n-1$:
$$
     a_k = \frac{1}{4} \lambda_k \mbox{   for   } k \le \alpha^4n , \  \ 
      a_k = \frac{1}{16} \frac{ \alpha^2 k }{\sqrt{n}}  \mbox{   for   }  k \ge \alpha^4 n.
$$
(Recall the definition of $\lambda_k$ from (\ref{varlam}).)
Then,
an application of the Cauchy-Schwarz inequality yields: for all $v$ of length one,
\begin{eqnarray*}
  \frac{1}{\sqrt{\beta}} \sum_{k=2}^n (-\widetilde{Z}_k) ( v_k^2 - v_{k-1}^2) \le
      {1  \over 4}  \sum_{k=1}^{n-1} 
                    \lambda_k (v_{k+1}  + v_k)^2                 
      +     {\alpha ^2 \over 4 \sqrt n}  \sum_{k=1}^n k v_k^2      
            +  \max _{1 \leq k \leq n-1}  {\widetilde{Z}_{k+1}^2 \over {\beta a_k}} 
\end{eqnarray*}
A similar estimate applies to $\sum_{k=1}^{n-1} Y_k v_k (v_{k+1} - v_k)$.  Accordingly,
$$
 L'(v) \le  \tilde{L}(v,U)  +  \max _{1 \leq k \leq n-1}  {\widetilde{Z}_{k+1}^2 \over {\beta a_k}}  +  \max _{1 \leq k \leq n-1}  {{Y}_{k}^2 \over {\beta a_k}} 
$$ 
with
$$
    \tilde{L}(v,U) =  \frac{1}{\sqrt{\beta}} \sum_{k=2}^n (- {U}_k) v_k^2  
                          -    \frac{1}{2} \sum_{k=1}^{n-1}   
                      \lambda_k (v_{k+1} + v_k)^2  
                 - { \alpha ^2\over {2 \sqrt n}} \sum_{k=1}^n k v_k^2. 
$$
Obviously, the arguments of step 1 apply to $\tilde{L}(v,U)$.

Finally, with $W_k$ either $\widetilde{Z}_{k+1}$ or $Y_k$, Lemmas 10 and 11 imply that
$$
  \PP \left( \max _{1 \leq k \leq n-1}  { W_k^2 \over {\beta a_k}}   \ge \eps \alpha^{4/3} \sqrt{n} \right) \le C  \sum_{k=1}^n e^{ -\beta \eps \alpha^{4/3} a_k /C},
$$ 
provided say $\eps \le 1$.  Since it may be assumed that $\alpha < 1/2$ (otherwise we are in the easy regime covered by Corollary 13), we 
have the bound $a_k = \lambda_k \le \sqrt{n}/C$ for $k \le \alpha^4 n \le n/2$ and so also
$$
 \sum_{k=1}^n e^{ - \varepsilon \alpha ^{4/3} \sqrt n a_k /C }
              \leq  \alpha ^4 n \, e^{- \beta \varepsilon \alpha ^{4/3} n  /C }
                  + { C \over \varepsilon \alpha ^{10/3}} \,  e^{ - \beta \varepsilon \alpha ^{22/3} n /C}, 
$$
by considering the sums over $k \le \alpha^4 n $ and $k > \alpha^4 n$ separately. If now $\eps \le \alpha^{44/3}$ 
(still keeping in mind that $\eps^{3/2} n \ge 1$), the right hand side is less than $C e^{-\beta n \eps^{3/2} /C}$.
Adding this new constraint on $\eps$ to those stated in (\ref{almost}) completes the proof.

\bigskip

\noindent{\bf{Acknowledgments }} 
The work of the first author was supported in part by
the French ANR GRANDMA, that of the second author by 
NSF grant DMS-0645756.  The second author also thanks the
Institut de Math\'ematiques de Toulouse, during a visit to which much of the present work 
was completed, for their hospitality.

\sc \noindent
Michel Ledoux \\ 
Institut de Math\'ematiques de Toulouse, \\
Universit\'e de Toulouse, F-31062 Toulouse, France.
\\ {\tt ledoux@math.univ-toulouse.fr}

\sc \bigskip \noindent Brian Rider \\  Department of
Mathematics,
\\ University of Colorado at Boulder, Boulder, CO 80309. \\{\tt
brian.rider@colorado.edu}


\begin{thebibliography}{99}

\bibitem{AMS}
{\sc Aida, S., Masuda, T., and Shigekawa, I.}
(1994) Logarithmic Sobolev inequalities and exponential integrability.
{\it J.\ Funct.\ Anal.} {\bf 126}, no.\ 1, 83-101.

\bibitem{BDJ}
{\sc Baik, J., Deift, P., Johansson, K. } (1999) On the
distribution of the length of the longest increasing
subsequence of random permutations. {\it J. Amer. Math.
Soc.}  {\bf 12},  no. 4, 1119--1178.

\bibitem {BDMLMZ}
{\sc Baik, J., Deift, P., McLaughlin, K., Miller, P. and Zhou, Z.} (2001)
Optimal tail estimates for directed last passage site percolation with geometric
random variables. {\it Adv. Theor. Math. Phys.} {\bf 5}, 1207--1250.

\bibitem{BF}{\sc Borodin, A., Forrester, P.} (2003)
Increasing subsequences and the hard-to-soft transition in matrix ensembles.
{\it J.\ Phys. A: Math and Gen.} {\bf 36}, no.\ 12, 2963-2982. 

\bibitem{DE} {\sc Dumitriu, I. and Edelman, A.} (2002)
Matrix models for beta ensembles. {\it J. Math. Phys.}
{\bf 43}, no. 11, 5830-5847.

\bibitem{ES}
{\sc Edelman, A., Sutton, B.} (2007) From random matrices
to stochastic operators. {\it J. Stat. Phys.}
{\bf 127}, no. 6, 1121-1165.

\bibitem{EK}
{\sc El Karoui, N.} (2003) On the largest eigenvalue of
Wishart matrices with identity covariance when $n$, $p$,
and $p/n \rightarrow \infty$.  {\it To appear, Bernoulli.}

\bibitem{FS}
{\sc Feldheim, O., Sodin, S.} (2009)
A universality result for the smallest eigenvalues of certain sample
covariance matrices.  {\it To appear, Geom.\ Funct.\ Anal.}

\bibitem{For}  {\sc  Forrester, P.}  
Log-gases and Random Matrices. {\it Textbook in press,} 2009.

\bibitem {FR}
{\sc Forrester, P. and Rains, E.}
Interrelations between orthogonal, unitary and symplectic matrix
ensembles (2001). In {\it Random Matrix Models and their Applications,}
voume 40 of {\it Math. Sci. Inst. Res. Publ.}, 171--207. Cambridge
University Press.

\bibitem{J}
{\sc Johansson, K.} (2000) Shape fluctuations and random
matrices. {\it  Comm. Math. Phys.} {\bf 209},  no. 2, 437-476.

\bibitem{JS}
{\sc Johnstone, I. M.} (2001)
 On the distribution of the largest eigenvalue in principal components analysis.
 {\it Ann. Statist.}  {\bf  29},  no. 2, 295-327.

\bibitem{L0}
{\sc Ledoux, M.} 
The Concentration of Measure Phenomenon.
Mathematical Surveys and Monographs {\bf 89}. Amer. Math. Soc. 2001.

\bibitem{L1}
{\sc Ledoux, M.} (2004)
Differential operators and spectral distributions of invariant ensembles from the
classical orthogonal polynomials: The continuous case. 
{\it Elect. Journal Probab.} {\bf 9}, 177-208.


\bibitem{L2}
{\sc Ledoux, M.} (2007) 
Deviation inequalities on largest eigenvalues.
Lecture Notes in Math. {\bf 1910}, 167-219. Springer.


\bibitem{L3}
{\sc Ledoux, M.} (2009)
A recursion formula for the moments of the Gaussian Orthogonal Ensemble. 
{\it  Annales Inst.\ H.\ Poincar\'e} {\bf 45}, 754-769.


\bibitem{RR}
{\sc  Ram\'{\i}rez, J., Rider, B.} (2009)
Diffusion at the random matrix hard edge.
 {\em Comm.\ Math.\ Physics}  {\bf 288},  no.\ 3, 887-906.


\bibitem{RRV}
{\sc  Ram\'{\i}rez, J., Rider, B.,  and Vir\'ag, B.} (2007)
Beta ensembles, stochastic Airy spectrum and a diffusion.
{\it  Preprint, arXiv:math.PR/0607331.}


\bibitem{Silv}
{\sc Silverstein, J.} (1985)
The smallest eigenvalue of a large dimensional Wishart matrix.
{\it Ann Probab}. {\bf 13}, no.\ 4, 1364-1368.

\bibitem{Sosh} 
{\sc Soshnikov, A.} (1999)
Universality at the edge of the spectrum in Wigner random matrices.
{\it Comm. Math. Phys.} {\bf 207},  697-733.

\bibitem{TaoVu}
{\sc Tao, T., and Vu, V.} (2009)
Random matrices: Universality of local eigenvalue statistics up to the edge.
{\it Preprint, arXiv:0908.1982.}


\bibitem{TW1} {\sc Tracy, C., and Widom, H.} (1994)
Level spacing distributions and the Airy kernel. {\it Comm.
Math. Phys.} {\bf 159}  no. 1, 151-174.

\bibitem{TW2} {\sc Tracy, C., and Widom, H. } (1996)
On orthogonal and symplectic matrix ensembles. {\it Comm.
Math. Phys.} {\bf 177}  no. 3, 727-754.

\bibitem{TW3} {\sc  Tracy, C., and Widom, H. } (2008)
Asymptotics in ASEP with Step Initial Condition.
{\it Preprint,  arXiv:0807.1713.}

\bibitem{Trotter}{\sc Trotter, H. F.} (1984)
Eigenvalue distributions of large Hermitian matrices;
Wigner's semicircle law and a theorem of Kac, Murdock, and
Szeg\H{o}. {\em Adv. in Math.} {\bf 54}  no. 1, 67-82

\end{thebibliography}
\end{document}